\documentclass[12pt, a4paper, twoside]{amsart}

\usepackage[T1]{fontenc}
\usepackage[applemac]{inputenc}
\usepackage{amsmath}
\usepackage{amsthm}
\usepackage{amsfonts}
\usepackage{bbm}
\usepackage[UKenglish]{babel}
\usepackage{enumerate}
\usepackage{bm}
\usepackage{ae}
\usepackage{amssymb}

\newcounter{case}
\newcommand{\case}{\refstepcounter{case}\emph{Case \thecase. }}

\theoremstyle{definition}

\theoremstyle{remark}
\newtheorem*{rem}{Remark}
\theoremstyle{plain}
\newtheorem{thm}{Theorem}
\newtheorem{lem}[thm]{Lemma}

\newtheorem{cor}[thm]{Corollary}

\makeatletter
\newcommand*{\house}[1]{%
	\mathord{%
		\mathpalette\@house{#1}%
	}%
}
\newcommand*{\@house}[2]{%
	\dimen@=\fontdimen8 %
	\ifx#1\scriptscriptstyle\scriptscriptfont
	\else\ifx#1\scriptstyle\scriptfont
	\else\textfont\fi\fi
	3 %
	\sbox0{%
		$#1%
		\vrule width\dimen@\relax
		\overline{%
			\kern2\dimen@
			\begingroup 
			#2%
			\endgroup
			\kern2\dimen@
		}%
		\vrule width\dimen@\relax
		\mathsurround=1.5\dimen@ 
		$%
	}%
	\ht0=\dimexpr\ht0-\dimen@\relax
	\dp0=\dimexpr\dp0+2\dimen@\relax
	\vbox{%
		\kern\dimen@ 
		\copy0 %
	}%
}
\makeatother

\usepackage[normalem]{ulem}
\usepackage{xcolor}

\begin{document}
	
	\title{Infinite products with algebraic numbers}
	
	\author{SIMON KRISTENSEN}
	
	\address{S. Kristensen, Department of Mathematics, Aarhus University, Ny Munkegade 118,
		DK-8000 Aarhus C, Denmark}
	
	\email{sik@math.au.dk}

	\author{MATHIAS L\O{}KKEGAARD LAURSEN}
	
	\address{M. L. Laursen, Department of Mathematics, Aarhus University, Ny Munkegade 118,
		DK-8000 Aarhus C, Denmark}
	
	\email{mll@math.au.dk}

	\begin{abstract}
	We obtain general criteria for giving a lower bound on the degree of numbers of the form $\prod_{n=1}^\infty \left(1+\frac{b_n}{\alpha_n}\right)$ or of the form $\prod_{m=1}^\infty \left(1+ \sum_{n=1}^\infty \frac{b_{n,m}}{\alpha_{n,m}}\right)$, where the $\alpha_n$ and $\alpha_{n,m}$ are assumed to be algebraic integers, and the $b_n$ and $b_{n,m}$ are natural numbers. In each case, we give a lower bound of the degree over the smallest extension of $\mathbb{Q}$ containing all algebraic numbers in the expression. The criteria obtained depend on growth conditions on the involved quantities.
	\end{abstract}
	
	\thanks{Research supported by the Independent Research Fund Denmark (Grant ref. 1026-00081B) and Aarhus University Research Fund (Grant ref. AUFF-E-2021-9-20).}
	
	\maketitle
	
\section{Introduction}

Proving that a comcrete number is irrational can be a difficult task. Proving transcencence results can be even more difficult. In the present paper, we are concerned with general criteria showing that a number represented in a certain way has lower bounded degree. The criteria are on parameters of the representation, and so the representation of the number will reveal arithmetical properties of the number itself. This study has a long history, and we begin by giving some relevant highlights.

In \cite{MR2851961}, Erd\H{o}s proved that if $\varepsilon>0$ is fixed and $\{a_{n}\}_{n=1}^\infty$ is an increasing sequence of positive integers satisfying $a_n\ge n^{1+\varepsilon}$ and
\begin{equation*}
	\limsup_{n\to\infty} a_n^{1/2^n} = \infty,
\end{equation*}
then the number $\sum_{n=1}^{\infty}\frac{1}{a_nc_n}$ is irrational for all sequences of positive integers $\{c_n\}_{n=1}^\infty$.
This result has since seen many generalizations, including criteria for irrationality of infinite products and continued fractions (see \cite{Kolouch+Novotny} for an overview).
Later, Andersen, Kristensen and Laursen \cite{MR4022087, MR4033869,Laursen} have provided criteria for getting a lower bound on the algebraic degree of series of reciprocals of algebraic integers as well as continued fractions with algebraic integers as partial coefficients.

This leaves the case of infinite products, which we deal with in this note.
In the assumptions for our theorems and in their proofs, we let $\house{\alpha}$ denote the \emph{house} of an algebraic number $\alpha$, i.e., the maximum modulus among $\alpha$ and its algebraic conjugates.

	
\begin{thm}
\label{thm:main1}
Let $D \in \mathbb{N}$, $\varepsilon > 0$, $a\in(0,1)$, $e\in\{-1,1\}$, let $\{b_n\}_{n=1}^\infty$ be a sequence of positive integers, and let $\{\alpha_n\}_{n=1}^\infty$ be a sequence of algebraic integers, such that $\house{\alpha_n}\,b_n \le \vert \alpha_n \vert 2^{(\log_2\vert \alpha_n \vert)^a}$. Suppose that $\vert \alpha_n\vert$ increases, and that $\vert \alpha_n \vert > n^{1+\varepsilon}$ for $n$ sufficiently large. Furthermore, we suppose $(\Re(\alpha_n/b_n) + 1/2)e \ge 0$ for all $n \in \mathbb{N}$ with strict inequality for infinitely many $n\in\mathbb{N}$.
Write $\mathbb{K}_0=\mathbb{Q}$, $\mathbb{K}_{n+1} = \mathbb{K}_n(\alpha_{n+1})$, $d_n = \deg_{\mathbb{K}_{n-1}} \alpha_n$ and $D_n = \prod_{i=1}^n d_i$. 
Finally, suppose that $|\alpha_n|^{1/(D^n \prod_{i=1}^{n-1}(D_i + d_i))}$ diverges in $\mathbb{R}$ as $n\to\infty$.
Write $\mathbb{K}=\mathbb{Q}(\alpha_1,\alpha_2,\ldots)$. Then
$$
\deg_\mathbb{K}\left(\prod_{n=1}^\infty \left( 1 + \frac{b_n}{\alpha_n}\right)\right) > D.
$$
\end{thm}
This theorem generalizes a result by Han\v{c}l and Kolouch \cite{MR2851961}, which restricted $\alpha_n$ to be positive integers and only covered the case $D=1$ and $|\alpha_n|^{1/(D^n \prod_{i=1}^{n-1}(D_i + d_i))}=\infty$. 
\cite{MR2851961} does, however, give a more lenient bound for $b_n$.
In our concluding remark we point to how one would get similarly relaxed bounds on $b_n$ for the present paper.

We also provide a proof for the below theorem regarding infinite products of infinite series, which generalizes another theorem by Han\v{c}l and Kolouch \cite{MR3150834}, with their version having $\alpha_n\in \mathbb{N}$ and $D=1$.
\begin{thm}
	\label{thm:main2}
	Let $D \in \mathbb{N}$, let $\varepsilon>0$, let $(b_{n,m})_{m.n\in\mathbb{N}}$ be an infinite array of positive integers, and let $(\alpha_{n,m})_{m,n\in\mathbb{N}}$ be an infinite array of algebraic integers.
	Suppose that $\vert \alpha_{n,1}\vert$ increases, and that for $n$ sufficiently large,
	\begin{align}
		\label{eq:linear lower bound}
		n^{1+\varepsilon} &\le \vert \alpha_{n,1} \vert,  	
		\\	\label{eq:bound on sum}
		\sum_{j=1}^{n} \left\lvert \frac{b_{n-j+1,j}}{\alpha_{n-j+1,j}} \right\rvert &\le \lvert \alpha_{n,1}\rvert^{-1+(\log\log\lvert\alpha\rvert)^{-3-\varepsilon}},	
		\\	\label{eq:bound on product}
		\prod_{j=1}^{n} \house{\alpha_{n-j+1,j}} &\le \lvert \alpha_{n,1}\rvert^{n+(\log\log\lvert\alpha\rvert)^{-3-\varepsilon}}.
	\end{align}
	Furthermore, we suppose that $\Re(\alpha_{n,m})\ge 0$ and $e\Im\alpha_{n,m}\ge 0$ for all pairs $(m,n)$, where $e\in\{-1,1\}$ is fixed.
	Write $\mathbb{K}_0 = \mathbb{Q}$, $\mathbb{K}_{n+1} = \mathbb{K}_n(\alpha_{1,n+1},\alpha_{2,n},\ldots,\alpha_{n+1,1})$,
	and $D_n = [\mathbb{K}_n : \mathbb{Q}]$.
	Finally, suppose that
	\begin{equation}
		\label{eq:limsup=infty}
		\limsup_{N \rightarrow \infty} \lvert\alpha_{N,1}\rvert^{\frac{1}{ D^{N}N! \prod_{n=1}^{N-1}D_n}} = \infty.
	\end{equation}
	Let $\mathbb{K}=\mathbb{Q}(\alpha_{m,n} : m,n\in\mathbb{N})$.
	Then
	$$
	\deg_{\mathbb{K}} \left(\prod_{m=1}^{\infty} \left(1+\sum_{n=1}^{\infty} \frac{b_{n,m}}{\alpha_{n,m}} \right)\right) > D.
	$$
	\end{thm}
	
\begin{rem}
	As will be evident from the proof, the restrictions on real and imaginary values of $\alpha_{n,m}/b_{n,m}$ are only there to ensure that the sequence $\left\{\prod_{m=1}^{N} \left(1+\sum_{n=1}^{N-m+1} \frac{b_{n,m}}{\alpha_{n,m}}\right)\right\}_{N=1}^\infty$ does not take the same value infinitely often and that the terms $\left(1+\sum_{n=1}^{\infty} \frac{b_{n,m}}{\alpha_{n,m}} \right)$ are non-zero.
	In fact, either of the following assumptions would also have been sufficient.
	We will prove this together with the theorem.
	\begin{itemize}
		\item $\Re (\frac{\alpha_{n,m}}{b_{n,m}})\ge-\frac{1}{2}$ for all sufficiently large $m+n$ with $>$ infinitely often, and $e\Im \alpha_{n,m}\ge |\Re(\alpha_{n,m})|$ for all $m,n$, where $e\in\{-1,1\}$ is fixed.
		\item $|\Im(\alpha_{n,m})| \le\Re(\alpha_{n,m})$ for all $m,n$.
		\item $X<1$, $\Re (\frac{\alpha_{n,m}}{b_{n,m}}) \le 0$, and $|\Im(\alpha_{n,m})|\le R|\Re(\alpha_{n,m})|$ for all $m,n$, where $X = \sup_{m\in\mathbb{N}}\{\sum_{n=1}^{\infty}\frac{b_{n,m}}{|\alpha_{n,m}|}\}$ and $R\in(0,1/X)$ are fixed.
	\end{itemize}
\end{rem}
	
\section{Auxiliary results}
	
We will make heavy use of Weil heights and Mahler measures of algebraic numbers. We recall the definitions. 

Let $\alpha$ be an algebraic number, let $K$ be a number field containing $\alpha$ and let $M_K$ denote the set of places of $K$. Then, the (Weil) height of $\alpha$ is defined as
$$
H(\alpha) = \prod_{\nu \in M_K} \max\{1,\vert \alpha \vert_\nu\}^{d_\nu/d},
$$
where $d = [K : \mathbb{Q}]$ and $d_\nu = [K_\nu : \mathbb{Q}_\nu]$, and where $K_\nu$ and $\mathbb{Q}_\nu$ denote the completions of the fields at the place $\nu$. With the normalisation in the exponent, the height becomes independent of the field $K$.

We will also need to define the Mahler measure of $\alpha$. For this purpose, suppose that $\alpha$ is algebraic of degree $d$ and let $\alpha_1 = \alpha, \alpha_2, \dots, \alpha_d$ denote the conjugates of $\alpha$. Finally, let $a_d$ denote the leading coefficient of the minimal polynomial of $\alpha$ defined over $\mathbb{Z}$. The Mahler measure of $\alpha$ is defined as
$$
M(\alpha) = \vert a_d \vert \prod_{i=1}^d \max\{1, \vert \alpha_i \vert\}.
$$
Here, the only place playing a role is the usual Archimedean one, i.e. the modulus in the complex plane.

The following wonderful result is classical, see e.g. \cite{MR1756786}.

\begin{thm}
\label{thm:height_measure}
For an algebraic number $\alpha$ of degree $d$,
$$
H(\alpha) = M(\alpha)^{1/d}.
$$
\end{thm}

The following lemma from \cite{MR4022087} relates heights and houses.

\begin{lem}
\label{lem:height_house_measure}
Let $\alpha$ be an algebraic integer of degree $d$. Then,
$$
H(\alpha) = M(\alpha)^{1/d} \le \house{\alpha} \le M(\alpha) = H(\alpha)^d.
$$
The inequalities are best possible.
\end{lem}

We will need to know that the height remains unchanged on taking the reciprocal. This is also classical, see \cite{MR1756786}.

\begin{lem}
\label{lem:reciprocal}
Let $\alpha$ be a non-zero algebraic number. Then, $H(\alpha) = H(1/\alpha)$.
\end{lem}

\cite{MR1756786} also provides bounds of the Weil height of sums and products of algebraic numbers.

\begin{lem}
\label{lem:elementary}
Let $n \in \mathbb{N}$, and let $\beta_1, \dots, \beta_n$ be algebraic numbers. Then,
$$
H\left(\sum_{i=1}^n \beta_i\right) \le 2^n \prod_{i=1}^n H(\beta_i), \quad \text{and } H\left(\prod_{i=1}^n \beta_i\right) \le \prod_{i=1}^n H(\beta_i).
$$
\end{lem}


Our proof depends critically on the Liouville--Mignotte inequality \cite{liouville, MR554376}, which is the following.

\begin{lem}
\label{lem:liouville}
Let $\alpha$ and $\beta$ be non-conjugate algebraic numbers. Then,
$$
\vert \alpha - \beta \vert \ge (2H(\alpha)H(\beta))^{-\deg(\alpha)\deg(\beta)}.
$$
\end{lem}

A nice proof can be found in \cite{MR2136100}. The following two lemmas are found in \cite{MR4022087}.

\begin{lem}\label{lem:series_upper_bound}
	Let  $\{a_n\}_{n=1}^\infty$ be an increasing sequence of real numbers such that $a_n > n^{1+\varepsilon}$ for some $\varepsilon>0$ and all $n\in\mathbb{N}$.
	Then, for all $N\in\mathbb{N}$,
	\[
	\sum_{n=N}^\infty \frac{1}{a_n} < \frac{2+\frac{1}{\varepsilon}}{a_N^{\varepsilon/(1+\varepsilon)}}.
	\]
\end{lem}

\begin{lem}\label{lem:aN+1_lower_bound}
	Let $\{a_n\}_{n=1}^\infty$ be a sequence of real numbers such that
	$$\limsup_{n \rightarrow \infty} a_n = \infty$$
	Then for infinitely many $N\in\mathbb{N}$,
	$$a_{N+1} > \bigg(1+\frac{1}{k^2}\bigg) \max_{1\le n\le N} a_n.$$
\end{lem}
The following three lemmas are taken from \cite{MR3150834}.
While the first two of the below lemmas assumed $\alpha_{n,1}$ to be integers in their original form, this property is never used in the proofs, so they remain valid in the present formulation.
The third lemma has been generalized slightly from \cite{MR3150834}, but the proof is the same.
\begin{lem}\label{lem:bound_series_general}
	Let $\varepsilon$ and $\alpha_{n,1}$ be given as in Theorem \ref{thm:main2}.
	Then, for $N$ sufficiently large,
	\begin{equation*}
		\sum_{n=N}^{\infty} |\alpha_{n,1}|^{-1+\frac{1}{\log^{3+\varepsilon} \log\lvert\alpha_{n,1}\rvert}}
		< |\alpha_{N,1}|^{-\frac{\varepsilon}{2(1+\varepsilon)}}
	\end{equation*}
\end{lem}
\begin{lem}\label{lem:bound_series_an>2n}
	Let $\varepsilon$ and $\alpha_{n,1}$ be given as in Theorem \ref{thm:main2} such that
	\begin{equation}\label{eq:an>2n}
		2^n<|\alpha_{n,1}|
	\end{equation}
	Then, for $N$ sufficiently large, 
	\begin{equation*}
		\sum_{n=N}^{\infty} |\alpha_{n,1}|^{-1+\frac{1}{\log^{3+\varepsilon} \log\lvert\alpha_{n,1}\rvert}}
		< |\alpha_{N,1}|^{-1+\frac{1}{\log^{3+\varepsilon/2} \log\lvert\alpha_{N,1}\rvert}}
	\end{equation*}
\end{lem}
\begin{lem}\label{lem:bound_prod_an_huge}
	Let $\delta\in[0,1)$, and let $D\in\mathbb{N}$, let $(D_n)_{n=1}^\infty$ be a sequence of natural numbers.
	Suppose $(a_n)_{n=1}^\infty$ is a non-decreasing sequence of positive real numbers such that
	\begin{equation}\label{eq:an_large}
		\limsup_{n \rightarrow \infty} a_n^{\frac{1}{D^n(n+\delta)!\prod_{i=1}^{n-1}D_i}} = \infty.
	\end{equation}
	Then, for infinitely many $N$, 
	\begin{equation}
		\label{eq:bound_max_an_huge}
		a_{N+1}^{\frac{1}{D^{N+1}(N+1+\delta)!\prod_{i=1}^{N}D_i}} > \bigg(1+\frac{1}{N^2}\bigg)\max_{1\le n\le N} a_n^{\frac{1}{D^n(n+\delta)!\prod_{i=1}^{n-1}D_i}}
	\end{equation}
	and
	\begin{equation}
		\label{eq:bound_prod_an_huge}
		a_{N+1} > \Bigg(\bigg(1+\frac{1}{N^2}\bigg)^{D^N (N+1+\delta)!\prod_{i=1}^{N-1}D_i} \prod_{n=1}^N a_n^{n+\delta}\Bigg)^{DD_N}.
	\end{equation}
\end{lem}

As some applications of Lemma \ref{lem:bound_series_an>2n} are a little opaque, we will state a consequence of it that is more easily applied. It follows immediately by adding infinitely many terms to the finite sum of the corollary and subsequently applying Lemma \ref{lem:bound_series_an>2n}.
\begin{cor}\label{cor:bound_series_an>2n}
Let $\varepsilon$ and $\alpha_{n,1}$ be given as in Theorem \ref{thm:main2} such that
	\begin{equation*}
		2^n<|\alpha_{n,1}|,
	\end{equation*}
	for $n \in [t, k]$ for infinitly many disjoint intervals  $[t,k]$.
	Then, for $t$ sufficiently large, 
	\begin{equation*}
		\sum_{n=t}^{k} |\alpha_{n,1}|^{-1+\frac{1}{\log^{3+\varepsilon} \log\lvert\alpha_{n,1}\rvert}}
		< |\alpha_{t,1}|^{-1+\frac{1}{\log^{3+\varepsilon/2} \log\lvert\alpha_{t,1}\rvert}}
	\end{equation*}
\end{cor}

Finally, we present another lemma that will be useful for proving Theorem \ref{thm:main2}.
\begin{lem}\label{lem:size_of_product}
	Let $(a_n)_{n=1}^\infty$ be a sequence of complex numbers such that $\prod_{n=1}^\infty (1+a_n)$ is absolutely convergent.
	Write
	\begin{equation*}
		C = \sup_{K\in\mathbb{N}} \prod_{n=1}^{K-1} |1+a_n|.
	\end{equation*}
	Then
	\begin{align*}
		\left|1 - \prod_{n=1}^\infty (1+a_n)\right| \le C\sum_{n=1}^{\infty}|a_n|.
	\end{align*}
\end{lem}
\begin{proof}
	Let $K\in\mathbb{N}$.
	We will then show that
	\begin{equation*}
		\left|1 - \prod_{n=1}^K (1+a_n)\right| \le C\sum_{n=1}^{K}|a_n|
	\end{equation*}
	If $K=1$, this is trivial.
	If $K>1$, it follows by induction upon noting
	\begin{align*}
		\left|1 - \prod_{n=1}^K (1+a_n)\right| & \le \left|1 - \prod_{n=1}^{K-1} (1+a_n)\right| + |a_K| \prod_{n=1}^{K-1} |1+a_n|
		\\&
		\le \left|1 - \prod_{n=1}^{K-1} (1+a_n)\right| + C|a_K|.
	\end{align*}
	The lemma then follows by letting $K$ tend to infinity.
\end{proof}

\section{Proof of Theorem \ref{thm:main1}}
The theorem follows from the following two lemmas
\begin{lem}\label{lem:lower bound 1}
	Let $D$, $d_n$, $D_n$, $a$, $\varepsilon$, $\alpha_n$, and $b_n$ be given as in Theorem \ref{thm:main1}, except that $|\alpha_n|^{\frac{1}{D^n\prod_{n=1}^{N-1}(d_i+D_i)}}$ need not diverge.
	Suppose $ \prod_{n=1}^{\infty}\left(1+\frac{b_n}{\alpha_n}\right)$ has degree at most $D$ over $\mathbb{K}$.
	Then 
	\begin{equation}
	\label{eq:lower bound 1}
	\liminf_{N \rightarrow \infty}  \left(2^{N^2\log_2^\alpha|\alpha_N|} \prod_{n=1}^N |\alpha_n|\right)^{DD_N} \sum_{n=N+1}^\infty \left\vert \frac{b_n}{\alpha_n}\right\vert =\infty.
	\end{equation}
\end{lem}

\begin{proof}
For $N \in \mathbb{N}$, let
\[
x = \prod_{n=1}^\infty \left( 1 + \frac{1}{\alpha_n}\right)
\quad\text{and}\quad
x_N =\prod_{n=1}^N \left( 1 + \frac{1}{\alpha_n}\right).
\]
By Lemmas \ref{lem:elementary} and \ref{lem:reciprocal},
\begin{align*}
	H(x-x_N) &\le 2H(x)\prod_{n=1}^N 2H(\alpha_n)H(1/b_n)
	\\&
	= 2^{N+1}H(x) \prod_{n=1}^N  H(\alpha_n) H(b_n).
\end{align*}
Appealing to Lemma \ref{lem:height_house_measure}, we then have
\begin{equation}\label{eq:HxN}
	H(x-x_N) \le 2^{N+1}H(x) \prod_{n=1}^N  \house{\alpha_n}\, b_n.
\end{equation}

A simple calculation shows that $|1+b_n/\alpha_n|-1$ is negative, 0, or positive when $\Re( \alpha_n/b_n)+1/2$ is negative, 0, or positive, respectively, while the bounds $|\alpha_n|\ge|\alpha_1|> 1$ and $b_n\le 2^{\log_2^a |\alpha_n|}$ ensure that $|b_n/\alpha_n|<1$ and thereby $x_N\ne 0$.
Hence, the restriction on $\Re(\alpha_n/b_n)$ implies that $\{|x_N|\}_{N=1}^\infty$ is monotonous but not constant, so that $x_N\ne x$.
Since $x-x_N$ must be algebraic due to $\deg_{\mathbb{K}_N} x = D<\infty$, we get from Lemma \ref{lem:liouville} with $\alpha=x-x_N$ and $\beta=0$ that
\begin{align*}
\vert x - x_N\vert & \ge \frac{1}{(2H(x-x_N))^{\deg(x-x_N)}}.
\end{align*}
Since clearly $\mathbb{K}=\bigcup_{n=1}^\infty \mathbb{K}_n$, $\deg_{\mathbb{K}} x = \deg_{\mathbb{K}_N} x$ for all sufficiently large $N$.
Then $x-x_N\in \mathbb{K}_N(x)$, and so
\begin{align*}
	\deg(x-x_N)\le [\mathbb{K}_N:\mathbb{Q}]\deg_{\mathbb{K}_N} x 
	\le D \prod_{n=1}^{N}[\mathbb{K}_N:\mathbb{K}_{N-1}]
	= DD_N.
\end{align*}
Recalling inequality \eqref{eq:HxN}, we continue the lower bound of $|x-x_N|$,
\begin{equation*}
	\vert x - x_N\vert \ge \left(2^{N+2} H(x) \prod_{n=1}^N \house{\alpha_n}\, b_n \right)^{-DD_N}.
\end{equation*}
Then applying the assumed upper bound of $\house{a_n}\,b_n$, we have
\begin{align}\nonumber
	\vert x - x_N\vert &
	\ge \left(2^{N+1} H(x) \prod_{n=1}^N 2^{2\log_2^a |\alpha_n|}|\alpha_n| \right)^{-DD_N}
	\\&\label{eq:x-xN lower}
	\ge \left(2^{2N\log_2^a |\alpha_N|} \prod_{n=1}^N |\alpha_n| \right)^{-DD_N},
\end{align}
for all sufficiently large $N$

To get an upper bound on $|x-x_N|$, let $K\ge N$. Then
\begin{align*}
	\left\lvert 1-\frac{x_{K+1}}{x_N}\right\rvert &\le \left\lvert 1-\frac{x_K}{x_N}\right\rvert + \left\lvert\frac{b_{K+1}}{\alpha_{K+1}}\right\rvert\left\lvert \frac{x_K}{x_N}\right\rvert
\end{align*}
Recalling that $|x_N|$ is monotonous and taking induction in $K$, we have
\begin{align*}
	\left\lvert 1-\frac{x_{K+1}}{x_N}\right\rvert &\le \sum_{n=N+1}^{K+1} \left\lvert \frac{x_K}{x_N}\right\rvert\left|\frac{b_n}{\alpha_n}\right|
	\le \max\left\{1,\left|\frac{x_{K}}{x_N}\right|\right\}\sum_{n=N+1}^{K+1} \left\lvert \frac{x_K}{x_N}\right\rvert\left|\frac{b_n}{\alpha_n}\right|
	\\&
	\le \max\left\{\frac{1}{|x_N|},\left|\frac{x_{K}}{x_N}\right|\right\}\sum_{n=N+1}^{K+1} \left\lvert \frac{x_K}{x_N}\right\rvert\left|\frac{b_n}{\alpha_n}\right|.
\end{align*}
Letting $K$ tend to infinity, we then get
\begin{equation*}
	|x-x_N| = |x_N|\left|1-\frac{x}{x_N}\right| \le \max\{1,|x|\} \sum_{n=N+1}^{\infty}\left\lvert\frac{b_n}{\alpha_n}\right\rvert.
\end{equation*}
Combining this with inequality \eqref{eq:x-xN lower}, we conclude
\begin{align*}
	&\left(2^{N^2\log_2^a|\alpha_n|} \prod_{n=1}^N |\alpha_n|\right)^{DD_N} \sum_{n=N+1}^{\infty}\left\lvert\frac{b_n}{\alpha_n}\right\rvert	
	\\&\qquad\quad
	\ge \left(2^{N^2\log_2^a|\alpha_n|} \prod_{n=1}^N |\alpha_n|\right)^{DD_N} \sum_{n=N+1}^{\infty}\left\lvert\frac{b_n}{\alpha_n}\right\rvert
	\\&\qquad\quad
	\ge \frac{2^{DD_N (N^2-2N)\log_2^a|\alpha_n|}}{\max\{1,|x|\}}
	\underset{N\to\infty}{\longrightarrow} \infty,
\end{align*}
and the proof is complete.
\end{proof}
\begin{lem}\label{lem:upper bound 1}
	Let $D$, $d_n$, $D_n$, $a$, $\varepsilon$, and $\alpha_n$ be given as in Theorem \ref{thm:main1}, and let $c\in(a,1)$. Then
	\begin{equation}
	\label{eq:contradiction2}
	\liminf_{N \rightarrow \infty}  \left(2^{N^2\log_2^\alpha|\alpha_N|} \prod_{n=1}^N |\alpha_n|\right)^{DD_N} \left\vert\sum_{n=N+1}^\infty  \frac{b_n}{\alpha_n}\right\vert = 0.
	\end{equation}
\end{lem}
\begin{proof}\setcounter{case}{0}
	Write
	\begin{equation*}
		a_n = |\alpha_n|,
		\quad\text{and}\quad H_n = a_n^{\frac{1}{D^n\prod_{i=1}^{n-1} (D_i+d_i)}}.
	\end{equation*}
	
	The case of
	$$\liminf_{n \rightarrow \infty} H_n<\limsup_{n \rightarrow \infty} H_n < \infty,$$
	is merely a special case of Lemma 12 from \cite{Laursen}, by noting that for all $c\in(a,1)$,
	\begin{equation*}
		2^{N^2 \log_2^a a_N} = 2^{N^2 D^{aN} \prod_{i=1}^{N-1} (D_i+d_i)^a \log_2^a H_n}
		< 2^{D^{cN}\prod_{i=1}^{N-1} (D_i+d_i)^c}.
	\end{equation*}
	
	Henceforth, we assume $\limsup_{n \rightarrow \infty} |\alpha_n|^{\frac{1}{D^n \prod_{i=1}^{n-1}(D_i + d_i)}} = \infty$.
	Recall the definitions of $a_n$ and $H_n$ above.
	Write further
\begin{align*}
	D_{n,\delta} = D_n+\delta
	\quad\text{and}\quad
	H_{n,\delta} = a_n^{\frac{1}{D^n\prod_{i=1}^{n-1} D_{i,\delta}+d_{i}}},
\end{align*}
for any $\delta\ge 0$. 
Note that
\begin{align}
	\nonumber
	\bigg(\max_{1\le n\le N}H_{n,\delta}&\bigg)^{D^{N+1} \prod_{i=1}^N D_{i,\delta} + d_i}
	\ge a_N^{DD_{N,\delta}} \left(\max_{1\le n\le N}H_{n,\delta}\right)^{D^{N+1} d_N \prod_{i=1}^{N-1} D_{i,\delta} + d_i}
	\\ \nonumber
	\ge& a_N^{DD_{N,\delta}} \left(\max_{1\le n\le N}H_{n,\delta}\right)^{D^{N+1} D_{N,\delta} \prod_{i=1}^{N-2} D_{i,\delta} + d_i} 
	\\& \nonumber
	\left(\max_{1\le n\le N}H_{n,\delta}\right)^{D^{N+1} d_N d_{N+1} \prod_{i=1}^{N-2} D_{i,\delta} + d_i}
	\\ \label{eq:Sn_power_lower_bound}
	\ge&\cdots\ge
	\left(\prod_{n=1}^N a_n^{D^{N-n}}\right)^{DD_{N,\delta}}
	\ge \prod_{n=1}^N a_n^{DD_{N,\delta}}.
\end{align}
Let $c\in(a,1)$ and notice that
\begin{align*}
	\log_2 \bigg(1+\frac{1}{N^2}\bigg) D^{N+1} \prod_{i=1}^N D_{i,\delta} + d_i 
	> D^{N+1} N^3 D_{N,\delta} \prod_{i=1}^{N-1}(D_i + d_i)^c,
\end{align*}
for all sufficiently large $N$. Combined with inequality \eqref{eq:Sn_power_lower_bound}, Lemma \ref{lem:aN+1_lower_bound} now implies that if $\limsup_{n\to\infty}H_{n,\delta}$, then
\begin{align}\nonumber
	a_{N+1} &\ge \bigg(1+\frac{1}{N^2}\bigg)^{D^{N+1} \prod_{i=1}^N D_{i,\delta} + d_i} \bigg(\max_{1\le n\le N}H_{n,\delta}\bigg)^{D^{N+1} \prod_{i=1}^N D_{i,\delta} + d_i}
	\\&\label{eq:aN+1}
	> \left(2^{D^{N} N^3  \prod_{i=1}^{N-1}(D_i + d_i)^c}\prod_{n=1}^N a_n\right)^{DD_{N,\delta}},
\end{align}
for infinitely many $N$.

We will split the proof into several cases, but before doing so, notice that if $a_n\ge 2^n$ for all sufficiently large $n$, then
\begin{align}
	\nonumber
	\sum_{n=N+1}^{\infty} \frac{b_n}{a_n} &\le \sum_{n=N+1}^{\lfloor \log_2 a_{N+1}+1\rfloor} \frac{2^{\log_2^a a_{n}}}{a_n}+ \sum_{n>\log_2 a_{N+1}+1}\frac{2^{\log_2^a a_{n}}}{a_n}
	\\&\nonumber
	\le \frac{\log_2 a_{N+1} + 1}{a_{N+1}}2^{\log_2^a a_{N+1}} + \sum_{n>\log_2 a_{N+1}+1}^{\infty} \frac{2^{n^a}}{2^{n}}
	\\&
	\le \frac{\log_2 a_{N+1} + 1}{a_{N+1}}2^{\log_2^a a_{N+1}} + C \frac{2^{\log_2^a a_{N+1}}}{a_{N+1}} 
	\le\frac{2^{2\log_2^a a_{N+1}}}{a_{N+1}},
	\label{eq:series_upper_bound_an>2n}
\end{align}
for a suitably fixed $C>0$ and all sufficiently large $N$.

\case
$a_n\ge 2^n$ for all sufficiently large $n$, and $\limsup_{n \rightarrow \infty} H_{n,\delta} = \infty$ for some $\delta>0$.
Fix such a $\delta$.
Combining inequalities \eqref{eq:aN+1} and \eqref{eq:Sn_power_lower_bound}, there are infinitely many $N$ satisfying
\begin{align*}
	\sum_{n=N+1}^{\infty} \frac{b_n}{a_n} &\le \left(2^{D^{N} N^3  \prod_{i=1}^{N-1}(D_i + d_i)^c}\prod_{n=1}^N a_n\right)^{-DD_{N,\delta}} 
	\\&\quad\ 
	\cdot2^{2\log_2^\alpha \left(2^{D^{N} N^3  \prod_{i=1}^{N-1}(D_i + d_i)^c}\prod_{n=1}^N a_n\right)^{DD_{N,\delta}}}
	\\&
	\le \prod_{n=1}^N a_n^{-DD_{N,\delta/2}},
\end{align*}
so that
\begin{align*}
	\left(2^{N^2\log_2^a a_N} \prod_{n=1}^N a_n\right)^{DD_N} \sum_{n=N+1}^\infty \frac{b_n}{a_n} &\le \prod_{n=1}^N a_n^{DD_{N,\delta/3}}\prod_{n=1}^N a_n^{-DD_{N,\delta/2}}
	\\&
	\le \prod_{n=1}^N a_n^{-DD_{N,\delta/6}},
\end{align*}

\noindent which becomes arbitrarily small as $N$ increases.


\case 
$a_n\ge 2^n$ for all sufficiently large $n$, but $\limsup_{n \rightarrow \infty} H_{n,\delta} < \infty$ for all $\delta>0$.
Recall $a<c<1$, and pick $\delta>0$ such that 
\begin{equation*}
	m+\delta\le m^{c/(2a)}
\end{equation*}
for all $m\ge 2$.
Then there must be some $C>0$ such that
\begin{equation*}
	\log_2 a_N \le C D^N \prod_{i=1}^{N-1}D_{i,\delta} + d_i 
	\le \prod_{i=1}^{N-1}(D_i + d_i)^{c/a},
\end{equation*}
for all sufficiently large $N$, and so
\begin{equation}\label{eq:log2an_upper_bound}
	2^{\log_2^a a_N} \le 2^{\prod_{i=1}^{N-1}(D_i + d_i)^c}
\end{equation}
Inserting this and inequality \eqref{eq:aN+1} into inequality \eqref{eq:Sn_power_lower_bound} now yields
\begin{align*}
	\sum_{n=N+1}^{\infty} \frac{b_n}{a_n} &\le \frac{2^{2\log_2^a a_{N+1}}}{a_{N+1}}
	\le \frac{2^{\log_2^a \left(\left(2^{D^{N} N^3  \prod_{i=1}^{N-1}(D_i + d_i)^c}\prod_{n=1}^N a_n\right)^{DD_{N}}\right)}}{\left(2^{D^{N} N^3  \prod_{i=1}^{N-1}(D_i + d_i)^c}\prod_{n=1}^N a_n\right)^{DD_{N}}}  
\end{align*}
for infinitely many $N$.
By inequality \eqref{eq:log2an_upper_bound},
\begin{align*}
	&\log_2^a \left(\left(2^{D^{N} N^3  \prod_{i=1}^{N-1}(D_i + d_i)^c}\prod_{n=1}^N a_n\right)^{DD_{N}}\right)
	\\&\qquad
	\le (DD_{N})^a \left(D^{N} N^3  \prod_{i=1}^{N-1}(D_i + d_i)^c + 2\sum_{n=1}^N \prod_{i=1}^{N-1}(D_i + d_i)^{c/a}\right)^a
	\\&\qquad
	\le (DD_{N})^a \left(2 D^{N} N^3  \prod_{i=1}^{N-1}(D_i + d_i)^{c/a}\right)^a
	\\&\qquad
	= 2^a D^{a(N+1)} D_N^a N^{3a}\prod_{i=1}^{N-1}(D_i + d_i)^c,
\end{align*}
Continuing our bound on $\sum_{n=N+1}^{\infty} \frac{b_n}{a_n}$, we now have
\begin{align*}
	\sum_{n=N+1}^{\infty} \frac{b_n}{a_n} &
	\le \frac{2^{2^a D^{a(N+1)} D_N^a N^{3a}\prod_{i=1}^{N-1}(D_i + d_i)^c}}{\left(2^{D^{N} N^3  \prod_{i=1}^{N-1}(D_i + d_i)^c}\prod_{n=1}^N a_n\right)^{DD_{N}}} 
	\\&
	\le \left(2^{2D^{N} N^2  \prod_{i=1}^{N-1}(D_i + d_i)^c}\prod_{n=1}^N a_n\right)^{-DD_{N}}
\end{align*}
Using this and inequality \eqref{eq:log2an_upper_bound}, we conclude
\begin{align*}
	\left(2^{N^2\log_2^a a_N} \prod_{n=1}^N a_n\right)^{DD_N} \sum_{n=N+1}^\infty \frac{b_n}{a_n}	&
	\le \frac{2^{N^2\prod_{i=1}^{N-1}(D_i + d_i)^c}}{2^{2D^{N} N^2  \prod_{i=1}^{N-1}(D_i + d_i)^c}}
	\\&\le \frac{1}{2^{D^{N} N^2  \prod_{i=1}^{N-1}(D_i + d_i)^c}},
\end{align*}

\case 
$a_{n} < 2^n$ infinitely often.

Let $A >1$ be a large number, and pick $k_1\in\mathbb{N}$ such that
\begin{align}\label{eq:Hancl22}
	H_{k_1} > 2^A.
\end{align}
Then pick $k_2\le k_1$ maximal such that
\begin{align}\label{eq:Hancl20}
	a_{k_2} < 2^{k_2}.
\end{align}
Notice that $k_1$ grows large when $A$ does since $\limsup_{n\to\infty}H_{n}=\infty$.
The case assumption then implies
\begin{align}\label{eq:k0bounds}
	k_2 \underset{A\to\infty}{\longrightarrow} \infty.
\end{align}

By Lemma \ref{lem:aN+1_lower_bound}, we may now pick $N\ge k_2$ minimal such that
\begin{equation*}
	H_{N+1} > \left(1+ \frac{1}{(N+1)^2}\right) \max_{k_2\leq j\leq N} H_j
\end{equation*}
Since $\prod_{n=1}^{\infty}\left(1+ \frac{1}{(n+1)^2}\right) < \infty$, it follows by the choices of $k_1$ and $k_2$ that $N<k_1$ when $A$ is large enough.
Notice that $N$ satisfies inequality \eqref{eq:aN+1} with $\delta=0$.

Since $a_n$ is increasing, it follows by the choices of $k_2$ and $N$ that
\begin{align*}
	\prod_{n=1}^{N} a_n &\le a_{k_2}^{k_2} \prod_{n=k_2+1}^{N}\left(1+\frac{1}{N^2}\right)^{n-k_2} a_k
	\le a_{k_2}^{N} \prod_{n=k_2+1}^{N}\left(1+\frac{1}{N^2}\right)^N
	\\&
	\le 2^{N^2} \prod_{n=k_2+1}^{\infty}\left(1+\frac{1}{N^2}\right)^N
\end{align*}

Since $\prod_{n=1}^{\infty} \left(1+\frac{1}{N^2}\right)<\infty$, it follows for large enough values of $A$ (and thereby $k_2$) that
\begin{align}\label{eq:prod_an}
	\prod_{n=1}^{N} a_n \le 2^{N^3}.
\end{align}

We now turn to estimating the infinite series. 
By maximality of $k_2$, we have that $a_n\ge 2^n$ for each $n\in[N+1,k_1)$.
Hence, by inequalities \eqref{eq:series_upper_bound_an>2n} and \eqref{eq:aN+1},
\begin{align*}
	\sum_{n=N+1}^{k_1-1} \frac{b_n}{a_n} &\le \sum_{n=N+1}^{k_1-1} \frac{b_n}{a_n} + \sum_{n=k_1}^{\infty}\frac{1}{2^n} \le \frac{2^{2\log_2^a a_{N+1}}}{a_{N+1}}
	\\&
	\le \frac{2^{2\log_2^a \left(2^{D^{N} N^3  \prod_{i=1}^{N-1}(D_i + d_i)^c}\prod_{n=1}^N a_n\right)^{DD_{N}}}}{\left(2^{D^{N} N^3  \prod_{i=1}^{N-1}(D_i + d_i)^c}\prod_{n=1}^N a_n\right)^{DD_{N}}}
\end{align*}
By applying inequality \eqref{eq:prod_an} in the exponent of the numerator, we find
\begin{align*}
	&\log_2 \left(\left(2^{D^{N} N^3  \prod_{i=1}^{N-1}(D_i + d_i)^c}\prod_{n=1}^N a_n\right)^{DD_{N}}\right)
	\\&\qquad\qquad
	\le DD_{N} \left(D^{N} N^3  \prod_{i=1}^{N-1}(D_i + d_i)^c + N^3 \right)
	\\&\qquad\qquad
	\le 2D^{N+1} D_N N^3 \prod_{i=1}^{N-1}(D_i + d_i)^c,
\end{align*}
and so
\begin{align}\nonumber
	\sum_{n=N+1}^{k_1-1} \frac{b_n}{a_n} &
	\le \frac{2^{4D^{a(N+1)} D_N^a N^{3a} \prod_{i=1}^{N-1}(D_i + d_i)^{ac}}}{\left(2^{D^{N} N^3  \prod_{i=1}^{N-1}(D_i + d_i)^c}\prod_{n=1}^N a_n\right)^{DD_{N}}}
	\\&\label{eq:sum_tail}
	\le \frac{1}{\left(2^{D^{N} N^2  \prod_{i=1}^{N-1}(D_i + d_i)^c}\prod_{n=1}^N a_n\right)^{DD_{N}}}.
\end{align}

Noticing $a_n/b_n\ge n^{1+\varepsilon/2}$, Lemma \ref{lem:series_upper_bound} and the bound on $b_n$ yield
\begin{align*}
	\sum_{n=k_1}^\infty  \frac{b_n}{a_n} \le \frac{2+\frac{2}{\varepsilon}}{(a_{k_1}/b_{k_1})^{\varepsilon/(2+\varepsilon)}}
	\le \frac{2(1+\frac{1}{\varepsilon}) 2^{\frac{\varepsilon}{2+\varepsilon} \log_2^a a_{k_1}} }{ a_{k_1}^{\varepsilon/(2+\varepsilon)}}
\end{align*}
By choice of $k_1$ and since $N<k_1$, we then have
\begin{align*}
	\sum_{n=k_1}^\infty  \frac{b_n}{a_n} &
	\le \frac{2(1+\frac{1}{\varepsilon}) 2^{\frac{\varepsilon}{2+\varepsilon} (AD^{k_1}\prod_{i=1}^{k_1-1}(D_i+d_i))^a }}{ 2^{\frac{\varepsilon}{2+\varepsilon}A D^{k_1}\prod_{i=1}^{k_1-1}(D_i+d_i)}}
	\\&
	\le \frac{1}{ 2^{\frac{\varepsilon}{3+\varepsilon} AD^{N+1}\prod_{i=1}^{N}(D_i+d_i)}}
\end{align*}
Together with inequality \eqref{eq:sum_tail}, we then have
\begin{align*}
	\sum_{n=N+1}^\infty \frac{b_n}{a_n} &
	= \sum_{n=N+1}^{k_1-1}\frac{b_n}{a_n} + \sum_{n=k_1}^\infty \frac{b_n}{a_n}
	\\&
	\le \frac{1}{\left(2^{D^{N} N^2  \prod_{i=1}^{N-1}(D_i + d_i)^c}\prod_{n=1}^N a_n\right)^{DD_{N}}} 
	\\&\quad\ 
	+ \frac{1}{ 2^{\frac{\varepsilon}{3+\varepsilon} AD^{N+1}\prod_{i=1}^{N}(D_i+d_i)}}.
\end{align*}

Hence,
\begin{align*}
	\left(2^{N^2\log_2^a a_N} \prod_{n=1}^N a_n\right)^{DD_N} \sum_{n=N+1}^\infty \frac{b_n}{a_n}	&
	\le \frac{2^{N^2\log_2^a a_N}}{2^{D^{N+1}D_N N^2  \prod_{i=1}^{N-1}(D_i + d_i)^c}}
	\\&\quad\
	+ \frac{2^{N^2\log_2^a a_N} \prod_{n=1}^N a_n}{2^{\frac{\varepsilon}{3+\varepsilon} AD^{N+1}\prod_{i=1}^{N}(D_i+d_i)}}.
\end{align*}
The first summand clearly tends to 0 as $A$ (and thereby $N$) grows large.
The other summand is estimated through inequality \eqref{eq:prod_an},
\begin{align*}
	\frac{2^{N^2\log_2^a a_N} \prod_{n=1}^N a_n}{2^{\frac{\varepsilon}{3+\varepsilon} AD^{N+1}\prod_{i=1}^{N}(D_i+d_i)}}
	\le \frac{2^{N^{2+3a} + N^3}}{2^{\frac{\varepsilon}{3+\varepsilon} AD^{N+1}\prod_{i=1}^{N}(D_i+d_i)}},
\end{align*}
which shows that also this summand tends to 0.
Thereby,
\begin{equation*}
	\left(2^{N^2\log_2^a a_N} \prod_{n=1}^N a_n\right)^{DD_N} \sum_{n=N+1}^\infty \frac{b_n}{a_n}
\end{equation*}
can be made arbitrarily small.

Since we have now covered all possible cases, the proof is complete.
\qedhere

%
%
%
\end{proof}

\begin{proof}[Proof Theorem \ref{thm:main1}]
	Since the conclusions of Lemmas \ref{lem:lower bound 1} and \ref{lem:upper bound 1} are quite clearly mutually exclusive, the theorem follows by comparing hypotheses.
\end{proof}
\section{Proof of Theorem \ref{thm:main2}}
For this section, we will write
\begin{align*}
	x &= \prod_{m=1}^{\infty} \left(1+\sum_{n=1}^{\infty} \frac{b_{n,m}}{\alpha_{n,m}} \right),
	\qquad
	x_N = \prod_{m=1}^{N} \left(1+\sum_{n=1}^{N-m+1} \frac{b_{n,m}}{\alpha_{n,m}} \right).
\end{align*}
The proof of Theorem \ref{thm:main2} and the remark following it will be split into the following three lemmas.
\begin{lem}\label{lem:Re>}
	Let $\alpha_{m,n}$ and $b_{m,n}$ be given as in Theorem \ref{thm:main2}, except for the restrictions on real and imaginary values, and suppose one of the following statements holds.
	\begin{enumerate}[I.]
		\item\label{item:Re>Main} $\Re (\frac{\alpha_{n,m}}{b_{n,m}})\ge 0$ and $e\Im \alpha_{n,m}\ge 0$ for all $m,n$.
		\item\label{item:Re>near-1/2} $\Re (\frac{\alpha_{n,m}}{b_{n,m}})\ge-\frac{1}{2}$ for all sufficiently large $m+n$ with $>$ infinitely often and that $e\Im \alpha_{n,m}\ge |\Re(\alpha_{n,m})|$ for all $m,n$.
		\item\label{item:Re>aroundR} $\Re(\alpha_{n,m})\ge |\Im(\alpha_{n,m})|$ for all $m,n$.
		\item\label{item:Re<main} $X<1$, $\Re (\frac{\alpha_{n,m}}{b_{n,m}}) \le \frac{-1}{2(1-XR)}$, and $|\Im(\alpha_{n,m})|\le R|\Re(\alpha_{n,m})|$ for all pairs $(m,n)$, where $X = \sup_{m\in\mathbb{N}}\{\sum_{n=1}^{\infty}\frac{b_{n,m}}{|\alpha_{n,m}|}\}$ and $R\in[1,1/X)$ are fixed.
	\end{enumerate}
	Then $\left|1+\sum_{n=1}^{\infty}b_{n,m}/\alpha_{n,m}\right| \ge C_0$ for a fixed number $C_0\in(0,1)$ that does not depend on $m$, and 
	the sequence $\{x_N\}_{N=1}^\infty$ does not contain the same number infinitely often.
\end{lem}
\begin{proof}
		We first consider the estimate $\left|1+\sum_{n=1}^{\infty}b_{n,m}/\alpha_{n,m}\right| \ge C_0$
		In statements \ref{item:Re>Main} and \ref{item:Re>aroundR}, $\Re(\alpha_{n,m}/b_{n,m})\ge 0$, which implies $\Re(b_{n,m}/\alpha_{n,m})\ge 0$, and so
		\begin{align*}
				\left|1+\sum_{n=1}^{\infty}\frac{b_{n,m}}{\alpha_{n,m}}\right| 
				\ge \Re\left(1+\sum_{n=1}^{\infty}\frac{b_{n,m}}{\alpha_{n,m}}\right)
				= 1 + \sum_{n=1}^{\infty}\Re\left(\frac{b_{n,m}}{\alpha_{n,m}}\right)
				\ge 1.
			\end{align*}
		As for statement \ref{item:Re>near-1/2}, the bound on the imaginary value implies
		\begin{align*}
			\left|\Im\left(\frac{b_{n,m}}{\alpha_{n,m}}\right)\right|
			= \frac{b_{n,m}|\Im(\alpha_{n,m})|}{|\alpha_{n,m}|^2} 
			\ge \frac{b_{n,m}\Re(\alpha_{n,m})}{|\alpha_{n,m}|^2}
			=\Re\left(\frac{b_{n,m}}{\alpha_{n,m}}\right).
		\end{align*}
		From this and the converse triangle inequality follows
		\begin{align*}
				\left|1+\sum_{n=1}^{\infty}\frac{b_{n,m}}{\alpha_{n,m}}\right|	&
				\ge \max\left\{
					\left|1+\sum_{n=1}^{\infty}\Re\left(\frac{b_{n,m}}{\alpha_{n,m}}\right)\right|,
					\left|\sum_{n=1}^{\infty}\Im\left(\frac{b_{n,m}}{\alpha_{n,m}}\right)\right|
				\right\}
				\\&
				\ge \max\left\{
				1-\sum_{n=1}^{\infty}\left|\Re\left(\frac{b_{n,m}}{\alpha_{n,m}}\right)\right|,
				\sum_{n=1}^{\infty}\left|\Re\left(\frac{b_{n,m}}{\alpha_{n,m}}\right)\right|
				\right\}
				\\&
				\ge \frac{1}{2}.
			\end{align*}
		For statement \ref{item:Re<main}, we again use the converse triangle inequality,
		\begin{align*}
				\left|1+\sum_{n=1}^{\infty}\frac{b_{n,m}}{\alpha_{n,m}}\right|	&
				\ge 1-X > 0.
			\end{align*}

	We now turn our attention to the sequence $\{x_N\}_{N=1}^\infty$.
	Consider
	\begin{align}\nonumber
		\frac{|x_N|}{|x_{N-1}|} &= \prod_{m=1}^{N} \Bigg| 1 + \frac{ \frac{b_{N-m+1,m}}{\alpha_{N-m+1,m}} }{1 + \sum_{n=1}^{N-m} \frac{b_{n,m}}{\alpha_{n,m}}} \Bigg|
		\\&\label{eq:xN/xN-1}
		= \prod_{m=1}^{N} \Bigg| 1 + \frac{1}{ \frac{\alpha_{N-m+1,m}}{b_{N-m+1,m}}\left(1+ \sum_{n=1}^{N-m} \frac{b_{n,m}}{\alpha_{n,m}} \right)} \Bigg|.
	\end{align}
	We will focus on the numbers
	\begin{equation*}
		\xi_{N,m} = \frac{\alpha_{N-m+1,m}}{b_{N-m+1,m}} \Bigg(1+\sum_{n=1}^{N-m} \frac{b_{n,m}}{\alpha_{n,m}}\Bigg).
	\end{equation*}
	Since
	\begin{align*}
		|1+\xi_{N,m}^{-1}|^2 &
		= \frac{(|\xi_{N,m}|^2+\Re(\xi_{N,m}))^2 + \Im(\xi_{N,m})^2}{|\xi_{N,m}|^4}
		\\&
		= 1 + \frac{1 + 2\Re(\xi_{N,m})}{|\xi_{N,m}|^2},
	\end{align*}
	it follows that the number $|1+1/\xi_{N,m}| - 1$ will be 0, negative, or positive exactly when the number $\Re(\xi_{N,m})+1/2$ is 0, negative, or positive, respectively.
	Hence, the proof will follow from equation \eqref{eq:xN/xN-1} if we can show that $\Re(\xi_{N,m}) + 1/2$ is either always non-negative or always non-positive and that it is non-zero infinitely often. 
	
	We calculate
	\begin{align}\nonumber
		\Re(\xi_{N,m}) &= \Re\Bigg(\frac{\alpha_{N-m+1,m}}{b_{N-m+1,m}} \Bigg(1+\sum_{n=1}^{N-m} \frac{b_{n,m}}{\alpha_{n,m}}\Bigg)\Bigg)
		\\&
		= \Re\bigg(\frac{\alpha_{N-m+1,m}}{b_{N-m+1,m}}\bigg) 
		+ \sum_{n=1}^{N-m} \frac{\Re\left(\frac{\alpha_{N-m+1,m}}{b_{N-m+1,m}}\right) \Re\left(\alpha_{n,m}\right)b_{n,m}}{\left|\alpha_{n,m}\right|^2}
		\nonumber\\&\quad\
		+ \sum_{n=1}^{N-m} \frac{\Im\left(\frac{\alpha_{N-m+1,m}}{b_{N-m+1,m}}\right) \Im\left(\alpha_{n,m}\right)b_{n,m}}{\left|\alpha_{n,m}\right|^2}
		\label{eq:xiNm}
	\end{align}
	If statement \ref{item:Re>Main} holds, then certainly $\xi_{N,m}\ge 0 > -1/2$ for all $N,m$, and we are done.
	
	If Statement \ref{item:Re>near-1/2} holds, then 
	\begin{align*}
		\Im\left(\frac{\alpha_{N-m+1,m}}{b_{N-m+1,m}}\right) \Im\left(\alpha_{n,m}\right)
		&= \left|\Im\left(\frac{\alpha_{N-m+1,m}}{b_{N-m+1,m}}\right)\right| \left|\Im\left(\alpha_{n,m}\right)\right|
		\\&
		\ge \left|\Re\left(\frac{\alpha_{N-m+1,m}}{b_{N-m+1,m}}\right)\right| \left|\Re\left(\alpha_{n,m}\right)\right|,
	\end{align*}
	so that equation \eqref{eq:xiNm} yields
	\begin{align*}
		\xi_{N,m} \ge \Re\left(\frac{\alpha_{N-m+1,m}}{b_{N-m+1,m}}\right) \ge -\frac{1}{2},
	\end{align*}
	with strict inequality for infinitely many pairs of indices $(N,m)$.
	
	If statement \ref{item:Re>aroundR} holds, we find $\xi_{N,m}\ge 0>-1/2$, using parallel arguments to those used for statement \ref{item:Re>near-1/2}.
	
	Finally, suppose that statement \ref{item:Re<main} holds.
	Then equation \eqref{eq:xiNm} implies
	\begin{align*}
		\Re(\xi_{N,m})&= -\left|\Re\left(\frac{\alpha_{N-m+1,m}}{b_{N-m+1,m}}\right)\right|
		\\&\quad\
		+ \sum_{n=1}^{N-m} b_{n,m}\frac{\left|\Re\left(\frac{\alpha_{N-m+1,m}}{b_{N-m+1,m}}\right)\right| \left|\Re\left(\alpha_{n,m}\right)\right| + \Im\left(\frac{\alpha_{N-m+1,m}}{b_{N-m+1,m}}\right) \Im\left(\alpha_{n,m}\right) }{\left|\alpha_{n,m}\right|^2}
		\\&
		\le -\left|\Re\left(\frac{\alpha_{N-m+1,m}}{b_{N-m+1,m}}\right)\right|
		\\&\quad\ 
		+ \max\left\{\left|\Re\left(\frac{\alpha_{N-m+1,m}}{b_{N-m+1,m}}\right)\right|, \left|\Im\left(\frac{\alpha_{N-m+1,m}}{b_{N-m+1,m}}\right)\right|\right\}\sum_{n=1}^{N-m}  \frac{b_{n,m}}{\left|\alpha_{n,m}\right|}
		\\&
		\le  -\left|\Re\left(\frac{\alpha_{N-m+1,m}}{b_{N-m+1,m}}\right)\right| + R\left|\Re\left(\frac{\alpha_{N-m+1,m}}{b_{N-m+1,m}}\right)\right| \sum_{n=1}^{N-m} \frac{b_{n,m}}{\left|\alpha_{n,m}\right|}
		\\&
		< -\left|\Re\left(\frac{\alpha_{N-m+1,m}}{b_{N-m+1,m}}\right)\right|\left(1 - RX \right) \le -\frac{1}{2},
	\end{align*}
	for all pairs $(N,m)$. Since $\left|\Im\left(\frac{\alpha_{N-m+1,m}}{b_{N-m+1,m}}\right)\right|\le R\left|\Re\left(\frac{\alpha_{N-m+1,m}}{b_{N-m+1,m}}\right)\right|$, it follows that $\left|\Re\left(\frac{\alpha_{N-m+1,m}}{b_{N-m+1,m}}\right)\right|\ge \sqrt{1+R^2}^{-1} \left|\frac{\alpha_{N-m+1,m}}{b_{N-m+1,m}}\right|$.
	By inequalities \eqref{eq:linear lower bound} and \eqref{eq:bound on sum}, this converges uniformly to $\infty$ as $N\to\infty$.
	Hence, $\Re(\xi_{N,m})<-1/2$ for all $m$ when $N$ is sufficiently large.
	This completes the proof.
\end{proof}
\begin{lem}\label{lem:lower bound 2}
	Let $D$, $d_{n,m}$, $D_n$, $\varepsilon$, $b_{m,n}$, and $\alpha_{m,n}$ be given as in Theorem \ref{thm:main2}, except that we do not assume equation \eqref{eq:limsup=infty} and that the restriction on real and imaginary values may be replaced by either of the three alternative restrictions posed in the remark following the theorem.
	Suppose $\deg_{\mathbb{K}} x \le D$.
	Then 
	\begin{equation*}
		\lim_{N \rightarrow \infty} \left(2^{N^2} \prod_{n=1}^{N} \lvert\alpha_{n,1}\rvert^{n+\frac{n+2}{\log^{3+\varepsilon} \log\lvert\alpha_{n,1}\rvert}} \right)^{DD_N} \sum_{n=N+1}^{\infty} \frac{|\alpha_{n,1}|^{\frac{1}{\log^{3+\varepsilon} \log\lvert\alpha_{n,1}\rvert}}}{|\alpha_{n,1}|}
		=\infty.
	\end{equation*}
\end{lem}
\begin{proof}
	Suppose that $\deg_{\mathbb{K}} x\le d$, which ensure that $x$ is algebraic.
	
	By Lemmas \ref{lem:elementary} and \ref{lem:reciprocal}, we have
	\begin{align*}
		H(x_N) &
		\le \prod_{m=1}^N \Bigg(2^{N-m+1}\prod_{n=1}^{N-m+1}b_{n,m}H\bigg(\frac{1}{\alpha_{n,m}}\bigg)\Bigg) 
		\\&
		= 2^{N(N+1)/2} \prod_{n=1}^{N}\prod_{j=1}^n b_{n-j+1,j}H(\alpha_{n-j+1,j})
	\end{align*}
	We then apply Lemma \ref{lem:height_house_measure} as well as inequalities \eqref{eq:bound on sum} and \eqref{eq:bound on product} to find
	\begin{align}\nonumber
		H(x_N) &
		\le 2^{N(N+1)/2} \prod_{n=1}^{N}\prod_{j=1}^n \lvert\alpha_{n,1}\rvert^{-1+\frac{1}{\log^{3+\varepsilon} \log\lvert\alpha_{n,1}\rvert}} \house{\alpha_{n-j+1,j}}^2
		\\&\nonumber
		\le 2^{N(N+1)/2} \prod_{n=1}^{N} \lvert\alpha_{n,1}\rvert^{-n+n\frac{1}{\log^{3+\varepsilon} \log\lvert\alpha_{n,1}\rvert}} \lvert\alpha_{n,1}\rvert^{ 2n + 2\frac{1}{\log^{3+\varepsilon} \log\lvert\alpha_{n,1}\rvert} } 
		\\&\label{eq:height_bound_xN}
		= 2^{N(N+1)/2} \prod_{n=1}^{N} \lvert\alpha_{n,1}\rvert^{n + \frac{n+2}{\log^{3+\varepsilon} \log\lvert\alpha_{n,1}\rvert}} .
	\end{align}
	
	By Lemma \ref{lem:Re>}, it follows that $x_N\ne x$ for all sufficiently large $N$, which means that we may apply Lemma \ref{lem:liouville} with $\alpha=x-x_N$ and $\beta=0$, leading to
	\begin{equation*}
		\lvert x-x_N\rvert \ge (2H(x-x_N))^{-\deg(x-x_N)}
		.
	\end{equation*}
	Since clearly $\mathbb{K}=\bigcup_{n=1}^\infty \mathbb{K}_n$, $\deg_{\mathbb{K}} x = \deg_{\mathbb{K}_N} x$ for all sufficiently large $N$.
	Then $x-x_N\in \mathbb{K}_N(x)$, and so
	\begin{align*}
		\deg(x-x_N)\le [\mathbb{K}_N:\mathbb{Q}]\deg_{\mathbb{K}_N} x 
		\le D \prod_{n=1}^{N}[\mathbb{K}_N:\mathbb{K}_{N-1}]
		= DD_N.
	\end{align*}
	Using this and inequality \eqref{eq:height_bound_xN}, we continue the above estimate on $|x-x_n|$ and find
	\begin{equation*}
		\lvert x-x_N\rvert
		\ge \left(2^{\frac{2}{3}N^2} \prod_{n=1}^{N} \lvert\alpha_{n,1}\rvert^{n+\frac{n+2}{\log^{3+\varepsilon} \log\lvert\alpha_{n,1}\rvert}} \right)^{-DD_N}.
	\end{equation*}
	To also get an upper bound of $|x-x_N|$, notice that the number
	\begin{align*}
		C &= \sup_{K,N\in\mathbb{N}_0} \Bigg\{
		\prod_{m=1}^{N} \bigg|1+ \frac{\sum_{n=N-m+2}^{\infty}\frac{b_{m,n}}{a_{m,n}}}{1+\sum_{n=1}^{N-m+1}\frac{b_{m,n}}{a_{m,n}}}\bigg|
		\prod_{m=N+1}^{N+K}\bigg|1+ \sum_{n=1}^{\infty}\frac{b_{m,n}}{a_{m,n}}\bigg|\Bigg\}
	\end{align*}
	is a finite, positive number. Lemma \ref{lem:size_of_product} now yields
	\begin{align*}
		\frac{|x-x_N|}{|x_N|} &=\bigg|
		1 - \prod_{m=1}^N \bigg(1+ \frac{\sum_{n=N-m+2}^{\infty}\frac{b_{m,n}}{a_{m,n}}}{1+\sum_{n=1}^{N-m+1}\frac{b_{m,n}}{a_{m,n}}}\bigg)
		\prod_{m=N+1}^\infty\bigg(1+ \sum_{n=1}^{\infty}\frac{b_{m,n}}{a_{m,n}}\bigg)
		\bigg|
		\\&
		\le C \bigg(
		\sum_{m=1}^N \bigg| \frac{\sum_{n=N-m+2}^{\infty}\frac{b_{m,n}}{a_{m,n}}}{1+\sum_{n=1}^{N-m+1}\frac{b_{m,n}}{a_{m,n}}}\bigg|
		+ \sum_{m=N+1}^\infty\bigg|\sum_{n=1}^{\infty}\frac{b_{m,n}}{a_{m,n}}\bigg|
		\bigg).
	\end{align*}
	By Lemma \ref{lem:Re>}, the numbers $\big|1-\sum_{n=1}^{\infty}\frac{b_{m,n}}{a_{m,n}}\big|$ have a uniform lower bound, $C_0$, while inequalities \eqref{eq:linear lower bound} and \eqref{eq:bound on sum} ensure that each $\lim_{m\to\infty} \sum_{n=1}^{\infty}\big|\frac{b_{m,n}}{a_{m,n}}\big| = 0$.
	Hence, the numbers $\big|1+\sum_{n=1}^{N-m+1}\frac{b_{m,n}}{a_{m,n}}\big|$ are bounded by $C_0/2$ for all sufficiently large $N$, and so
	\begin{align*}
		\frac{|x-x_N|}{|x_N|} &
		\le C \bigg(
		\frac{2}{C_0}\sum_{m=1}^N \bigg| \sum_{n=N-m+2}^{\infty}\frac{b_{m,n}}{a_{m,n}}\bigg|
		+ \sum_{m=N+1}^\infty\bigg|\sum_{n=1}^{\infty}\frac{b_{m,n}}{a_{m,n}}\bigg|
		\bigg)
		\\&
		\le\frac{2C}{C_0} \bigg(\sum_{m=1}^N  \sum_{n=N-m+2}^{\infty}\bigg|\frac{b_{m,n}}{a_{m,n}}\bigg|
		+ \sum_{m=N+1}^\infty\bigg|\sum_{n=1}^{\infty}\frac{b_{m,n}}{a_{m,n}}\bigg|
		\bigg)
		\\&
		= \frac{2C}{C_0} \sum_{n+m\ge N+2}^{\infty}\bigg|\frac{b_{m,n}}{a_{m,n}}\bigg|
		=\frac{2C}{C_0} \sum_{n=N+1}^{\infty} \sum_{j=1}^{n}\bigg|\frac{b_{n-j+1,j}}{a_{n-j+1,j}}\bigg|.
	\end{align*}
	by also applying the triangle inequality in the second estimate.
	By inequality \eqref{eq:bound on sum}, we now have
	\begin{align*}
		|x-x_N|&\le \frac{2C}{C_0}|x_N|\sum_{n=N+1}^{\infty} |\alpha_{n,1}|^{-1+\frac{1}{\log^{3+\varepsilon} \log\lvert\alpha_{n,1}\rvert}}
		\\&\le C' \sum_{n=N+1}^{\infty} |\alpha_{n,1}|^{-1+\frac{1}{\log^{3+\varepsilon} \log\lvert\alpha_{n,1}\rvert}},
	\end{align*}
	for a suitable constant $C'>0$ that does not depend on $N$.
	Recalling the lower bound on $|x-x_N|$ found above, we conclude
	\begin{align*}
		&\Bigg(2^{N^2} \prod_{n=1}^{N} \lvert\alpha_{n,1}\rvert^{n+\frac{n+2}{\log^{3+\varepsilon} \log\lvert\alpha_{n,1}\rvert}} \Bigg)^{DD_N} \sum_{n=N+1}^{\infty} \frac{|\alpha_{n,1}|^{\frac{1}{\log^{3+\varepsilon} \log\lvert\alpha_{n,1}\rvert}}}{|\alpha_{n,1}|}
		\\&	\qquad\qquad
		\ge \frac{2^{D D_N \frac{N^2}{3}}}{C'}
		\underset{N\to\infty}{\longrightarrow} \infty.
		\qedhere
	\end{align*}
\end{proof}
\begin{lem}\label{lem:upper bound 2}
	Let $D$, $d_{n,m}$, $D_N$, $\varepsilon$, and $\alpha_{n,1}$ be given as in Theorem \ref{thm:main2}.
	Then 
	\begin{equation*}
		\liminf_{N \rightarrow \infty} \left(2^{N^2} \prod_{n=1}^{N} \lvert\alpha_{n,1}\rvert^{n+\frac{n+2}{\log^{3+\varepsilon} \log\lvert\alpha_{n,1}\rvert}} \right)^{DD_N} \sum_{n=N+1}^{\infty} \frac{|\alpha_{n,1}|^{\frac{1}{\log^{3+\varepsilon} \log\lvert\alpha_{n,1}\rvert}}}{|\alpha_{n,1}|}
		=0.
	\end{equation*}
\end{lem}
\begin{proof}[Proof (original).]\setcounter{case}{0}
	To simplify notation, we introduce $a_n=|\alpha_{n,1}|$ and
	\begin{equation*}
		Z_N = \left(2^{N^2} \prod_{n=1}^{N} a_n^{n+\frac{n+2}{\log^{3+\varepsilon} \log a_n}} \right)^{DD_N} \sum_{n=N+1}^{\infty} \frac{a_n^{\frac{1}{\log^{3+\varepsilon} \log a_n}}}{a_n},
	\end{equation*}
	so that our aim is to prove that $Z_N$ has no positive lower bound.
	We now split into four cases.
	
\case\label{case:huge_not_tiny} Assume that equation \eqref{eq:an>2n} holds for all sufficiently large $N$ and that there is a real number $0<\delta<1$ such that $a_n$ and $\delta$ satisfy equation \eqref{eq:an_large}.
	By Lemma \ref{lem:bound_prod_an_huge} and equation \eqref{eq:limsup=infty}, we then have infinitely many $N$ so that
	\begin{equation}
		\label{eq:aN+1_lower_bound_huge}
		a_{N+1} > \Bigg(\bigg(1+\frac{1}{N^2}\bigg)^{D^N (N+1+\delta)!\prod_{i=1}^{N-1}D_i} \prod_{n=1}^N a_n^{n+\delta}\Bigg)^{DD_N}.
	\end{equation}
	Then
	\begin{align}\label{eq:aN+1_lower_bound_huge2}
		a_{N+1} &> \bigg(1+\frac{1}{N^2}\bigg)^{D^{N+1} (N+1+\delta)!\prod_{i=1}^{N}D_i}
		\ge 2^{N^{5+\varepsilon}DD_N}
	\end{align}
	and
	\begin{align*}
		\log \log a_{N+1} &
		\ge 
		\log \left( DD_N\log\left( (1+N^{-2})^{D^N(N+1+\delta)! \prod_{i=1}^{N-1}D_i} 
		\right) \right)
		\\&
		\ge  \log\left( \frac{(N+1+\delta)!}{2N^2} D^{N+1} \prod_{i=1}^{N}D_i \right)
		> \log((N-1+\delta)!)
		\\&
		\ge \frac{(N-1+\delta)\log(N-1+\delta)}{2}
		\ge \frac{N\log N}{3} + \log 2.
	\end{align*}
	Using the latter after applying Lemma \ref{lem:bound_series_an>2n} and inequality \eqref{eq:aN+1_lower_bound_huge}, we have
	\begin{align*}
		&\sum_{n=N+1}^{\infty} \frac{a_n^{\frac{1}{\log^{3+\varepsilon} \log a_n}}}{a_n}\left(\prod_{n=1}^{N} a_n^{\frac{n+2}{\log^{3+\varepsilon} \log a_n}}\right)^{DD_N} 
		\\&\qquad
		< a_{N+1}^{ -1+ \frac{1}{\log^{3+\varepsilon/2}\log a_{N+1}} + \frac{2}{\log^{3+\varepsilon}\log a_{N+1}}}
		< a_{N+1}^{ -1+ \frac{2}{\log^{3+\varepsilon/2}\log a_{N+1}}}
		\\&\qquad
		< a_{N+1}^{ -1+ \left(\frac{N\log N}{3}\right)^{-3-\varepsilon/2} }.
	\end{align*}
	By equation \eqref{eq:aN+1_lower_bound_huge2}, we have
	\begin{align*}
		a_{N+1}^{\left(\frac{N\log N}{3}\right)^{-3-\varepsilon/2} - N^{-3-\varepsilon/2}} &< a_{N+1}^{N^{-3-\varepsilon}} < 2^{-N^2DD_N},
	\end{align*}
	and so
	\begin{align}
		\label{eq:from_an_never_tiny}
		&\left(2^{N^2}\prod_{n=1}^{N} a_n^{\frac{n+2}{\log^{3+\varepsilon} \log a_n}}\right)^{DD_N}\sum_{n=N+1}^{\infty} \frac{a_n^{\frac{1}{\log^{3+\varepsilon} \log a_n}}}{a_n} < a_{N+1}^{-1 + N^{-3-\varepsilon/2}}
		\\&\qquad\qquad\nonumber
		\le \Bigg(\prod_{n=1}^N a_n^{n+\delta}\Bigg)^{DD_N(-1 + N^{-3-\varepsilon/2})},
	\end{align}
	Thus,
	\begin{align*}
	Z_N &= \Bigg( 2^{N^2} \prod_{n=1}^{N}  a_n^{n+\frac{n+2}{\log^{3+\varepsilon} \log a_n}}  \Bigg)^{DD_N} \sum_{n=N+1}^{\infty} \frac{a_n^{\frac{1}{\log^{3+\varepsilon} \log a_n}}}{a_n}
		\\& 
		< \Bigg( \prod_{n=1}^N a_n^{-\delta + (n+\delta)N^{-3-\varepsilon/2}} \Bigg)^{DD_N}
		\\& 
		\le \Bigg( \prod_{n=1}^{N} a_n^{-\delta/2}\Bigg)^{DD_N}.
	\end{align*}
	As there are infinitely many such $N$, this completes the proof in the present case.
	
\case Assume that equation \eqref{eq:an>2n} holds for all sufficiently large $N$ and that there is no real number $0<\delta<1$ such that $a_n$ and $\delta$ satisfy equation \eqref{eq:an_large}.
	For all $0<\delta<1$, we then have
	\begin{equation}\label{eq:an_not_huge}
		a_n < 2^{D^n(n+\delta)!\prod_{i=1}^{n-1} D_i},
	\end{equation}
	for all sufficiently large $n$.
	
	By Lemma \ref{lem:bound_prod_an_huge} and equation \eqref{eq:limsup=infty}, we have infinitely many $N$ so that
	\begin{equation}
		\label{eq:aN+1_lower_bound_not_huge}
		a_{N+1} > \Bigg(\bigg(1+\frac{1}{N^2}\bigg)^{D^N (N+1)!\prod_{i=1}^{N-1}D_i} \prod_{n=1}^N a_n^{n}\Bigg)^{DD_N}.
	\end{equation}
	Notice that all arguments leading to \eqref{eq:from_an_never_tiny} in Case \ref{case:huge_not_tiny} remain valid when we replace $\delta$ by 0.
	Hence,
%
	equation \eqref{eq:aN+1_lower_bound_not_huge} implies
	\begin{align*}
		&\left(2^{N^2}\prod_{n=1}^{N} a_n^{\frac{n+2}{\log^{3+\varepsilon} \log a_n}}\right)^{DD_N}\sum_{n=N+1}^{\infty} \frac{a_n^{\frac{1}{\log^{3+\varepsilon} \log a_n}}}{a_n}
		< a_{N+1}^{-1 + N^{-3-\varepsilon/2}}
		\\&\qquad\quad
		< \Bigg( \bigg( 1+\frac{1}{N^2} \bigg)^{ (N+1)!}\prod_{n=1}^N a_n^{n} \Bigg)^{ DD_N\left(-1+ N^{-3-\varepsilon/2}\right)}.
	\end{align*}
	Let $\delta>0$ be sufficiently small.
	When the above $N$ grow sufficiently large, equation \eqref{eq:an_not_huge} and the fact that $(1+1/N)^{N}>2$ imply
	\begin{align*}
		Z_N&=\Bigg( 2^{N^2} \prod_{n=1}^{N} a_n^{n+\frac{n+2}{\log^{3+\varepsilon} \log a_n}}  \Bigg)^{DD_N} \sum_{n=N+1}^{\infty} \frac{a_n^{\frac{1}{\log^{3+\varepsilon} \log a_n}}}{a_n}
		\\&
		< \Bigg( \bigg(1+\frac{1}{N^2}\bigg)^{(N+1)!\left( -1+ N^{-3-\varepsilon/2}\right)} \prod_{n=1}^{N} a_n^{n N^{-3-\varepsilon/2}}  \Bigg)^{DD_N}
		\\&
		\le \Bigg( 2^{N!\left( -1+ N^{-3-\varepsilon/2}\right)} \prod_{n=1}^{N} 2^{(n+\delta)!\frac{n}{N^{3+\varepsilon/2}}}  \Bigg)^{DD_N}
		\\&
		\le \big(2^{DD_N}\big)^{\frac{N^2(N+\delta)!+N!}{N^{3+\varepsilon/2}} - N!}
		< 2^{-N},
	\end{align*}
	and so the proof is complete in this case.
	
\case  Assume that
\begin{equation}\label{eq:an<=2n}
	a_n\le 2^n
\end{equation}
holds for infinitely many $N$ and that there is a real number $0<\delta<1$ such that $a_n$ and $\delta$ satisfy equation \eqref{eq:an_large}.

We fix an $A >0$. By \eqref{eq:an_large}, there is an $n \in \mathbb{N}$ such that 
\[
a_n^{\frac{1}{D^n (n+\delta)! \prod_{i=1}^{n-1}D_i}} > A,
\]
so we may pick $k_1$ minimal with this property. We then choose $k_2 < k_1$ maximal such that $a_{k_2} \le 2^{k_2}$ by \eqref{eq:an<=2n}. If no such $k_2$ exists, we increase $A$ until it does. This is possible since $k_1$ tends to infinity with $A$ and since \eqref{eq:an<=2n} is satisfied for infinitely many indices.

Now, 
\begin{equation}\label{eq:a_k1>}
a_{k_1} > A^{D^n (n+\delta)! \prod_{i=1}^{n-1}D_i} = 2^{\log_2(A) D^n (n+\delta)! \prod_{i=1}^{n-1}D_i},
\end{equation}
so there is an $n < k_1$ with 
\begin{equation}\label{eq:a_t>}
a_{n+1} >  2^{D^{n+1} (n+1+\delta)! \prod_{i=1}^{n}D_i}.
\end{equation}
We pick $N \in [k_2,k_1)$ minimal such that the latter holds. Such an index exists since
\[
a_{k_2} \le 2^{k_2} \le 2^{D^{k_2} (k_2+\delta)! \prod_{i=1}^{k_2-1}D_i},
\]
since $a_n$ is increasing, and since $a_{k_1}$ satisfies \eqref{eq:a_k1>}. Note that as $A$ increases, both $k_1$ and $k_2$ will increase. Hence, $N$ will tend to infinity as $A$ tends to infinity.

Consider first the product,
\[
\prod_{n=1}^{N} a_n^{n + \frac{n+2}{\log^{3+\varepsilon}\log a_n}} = \left(\prod_{n=1}^{k_2} a_n^{n + \frac{n+2}{\log^{3+\varepsilon}\log a_n}}\right) \left(\prod_{n=k_2+1}^{N} a_n^{n + \frac{n+2}{\log^{3+\varepsilon}\log a_n}}\right) = M_1 \cdot M_2.
\]
By choice of $k_2$, since $a_n$ is an increasing sequence, 
\[
M_1 \le \prod_{n=1}^{k_2} 2^{n^2 + n \frac{n+2}{\log^{3+\varepsilon}\log 2^n}} \le \prod_{n=1}^{k_2} 2^{n\left(k_2 + \frac{k_2+2}{\log^{3+\varepsilon}\log 2^{k_2}}\right)} \le 2^{k_2^3},
\]
by carrying out the multipliction and noticing the triangular number in the exponent. 

Now, for $M_2$ we have by \eqref{eq:a_t>}
\begin{alignat*}{2}
M_2 &\le \prod_{n=k_2+1}^{N} \left(2^{D^n(n+\delta)!\prod_{i=1}^{n-1}D_i}\right)^{n + \frac{n+2}{\log^{3+\varepsilon} \log a_n}}\\ 
&\le \prod_{n=k_2+1}^{N} 2^{\left(D^n(n+\delta)!\prod_{i=1}^{n-1}D_i\right)\left(n+ \frac{n+2}{\log^{3+\varepsilon}(D^{n+1}(n+1+\delta)!(\prod_{i=1}^n D_i) \log 2)}\right)}\\
&\le 2^{\left(D^{N}(N+\delta)!\prod_{i=1}^{N-1}D_i\right)\left(N+ \frac{1}{N^2}\right)} \prod_{n=k_2+1}^{N-1} 2^{\left(D^n(n+\delta)!\prod_{i=1}^{n-1}D_i\right)\left(n+ \frac{1}{n^2}\right)}\\
&\le \left(2^{(N+\delta)!\left(N+ \frac{1}{N^2}\right)} \prod_{n=k_2+1}^{N-1} 2^{(n+\delta)!\left(n+ \frac{1}{n^2}\right)}\right)^{D^{N}\prod_{i=1}^{N-1}D_i},
\end{alignat*}
where we have bounded $\frac{n+2}{\log^{3+\varepsilon}D^{n+1}(n+1+\delta)!\prod_{i=1}^n D_i \log 2}$ rather brutally by $1/n^2$. To continue,
\[
2^{(N+\delta)!\left(N+ \frac{1}{N^2}\right)} = 2^{(N+1+\delta)! - \left(1+\delta - \frac{1}{N^2}\right)(N+\delta)!}.
\]
Thus, for $N$ large enough, which we can ensure by increasing $A$, the term
\[
2^{- \left(1+\delta - \frac{1}{N^2}\right)(N+\delta)!}
\]
will cancel out the product over the remaining $n$'s as well as the term coming from $M_1$.
Thus,
\[
M_1 M_2 \le 2^{\left((N+1+\delta)! - \frac{\delta}{2}(N+\delta)!\right) D^{N}\prod_{i=1}^{N-1}D_i}.
\]

Now, consider the sum
\[
\sum_{n=N+1}^{\infty} \frac{a_n^{\frac{1}{\log^{3+\varepsilon} \log a_n}}}{a_n} = \sum_{n=N+1}^{k_1-1} \frac{a_n^{\frac{1}{\log^{3+\varepsilon} \log a_n}}}{a_n} +\sum_{n=k_1}^{\infty} \frac{a_n^{\frac{1}{\log^{3+\varepsilon} \log a_n}}}{a_n} =S_1 + S_2.
\]
By Corollary \ref{cor:bound_series_an>2n}, choice of $N$,  and \eqref{eq:a_t>}, 
\[
S_1 \le a_{N+1}^{\frac{1}{\log^{3+\varepsilon/2}\log a_{N+1}} - 1} \le 2^{D^{N+1} (N+1+\delta)! (\frac{1}{(N+1)^3} - 1)\prod_{i=1}^{N+1} D_i},
\]
by a brutal estimate in the exponent. For $S_2$, Lemma \ref{lem:bound_series_general} together with \eqref{eq:a_k1>} gives us that 
\[
S_2 \le a_{k_1}^{-\frac{\varepsilon}{2(1-\varepsilon)}} \le 2^{-(\log_2 A) \frac{\varepsilon}{2(1-\varepsilon)} D^{N+1} (N+1+\delta)! \prod_{i=1}^{N}D_i}.
\]

The upshot is that 
\[
Z_{N} \le \left(2^{N^2} M_1 M_2\right)^{DD_{N}}(S_1 + S_2).
\]
But this evidently tends to $0$ as $A$ -- and hence $N$ -- increases, by inserting all the above estimates. This completes the proof in this case.

	
\case 
Assume that equation \eqref{eq:an<=2n} holds for infinitely many $N$ and that there is no real number $0<\delta<1$ such that $a_n$ and $\delta$ satisfy equation \eqref{eq:an_large}.
For all $0<\delta<1$, equation \eqref{eq:an_not_huge} holds for all sufficiently large $n$.
Note that by the \emph{limsup} assumption in Theorem \ref{thm:main2}, we instead have equation \eqref{eq:an_large} if we let $\delta = 0$.

Let $\delta>0$ be fixed, and let $A$ be sufficiently large. By inequality \eqref{eq:an_not_huge},
\begin{equation}
\label{eq:Sum_0}
a_n \le 2^{D^n (n + \delta)! \prod_{i=1}^{n-1}D_i} 
\end{equation}
holds for all sufficiently large $n \in \mathbb{N}$.
Pick $k_1 \in \mathbb{N}$ to be minimal with the property that
\[
a_{k_1} > A^{D^{k_1}  k_1! \prod_{i=1}^{k_1-1}D_i}.
\]
Now, choose $k_2 < k_1$ to be maximal so that $a_{k_2} \le 2^{k_2}$. Should no such $k_2$ exist, then choose $A$ larger. Note also that for $A$ increasing, both $k_1$ and $k_2$ must increase.

As before, we now choose $N \in [k_2,k_1)$ such that a large jump must occur at this place. Concretely, we use Lemma \ref{lem:aN+1_lower_bound} and inequality \eqref{eq:an_large} to let $N\ge k_2$ be minimal with
\[
a_{N+1}^{\frac{1}{D^{N+1}  (N+1)! \prod_{i=1}^{N}D_i}} > \left(1+\frac{1}{(N+1)^{1+\varepsilon/4}}\right) \max_{j=k_2, \dots, N}  a_j^{\frac{1}{D^j  j! \prod_{i=1}^{j-1}D_i}}.
\]
Since $\prod_{n=1}^{\infty}\left(1+\frac{1}{(N+1)^{1+\varepsilon/4}}\right)<\infty$, we must have $N<k_1$ when $k_1$ is sufficiently large.

We claim that $Z_{N}$ tends to zero with $N$, which again suffices, as both $k_1$ and $k_2$ tend to infinity with $A$, so that a subsequence of left hand sides in the lemma will tend to zero.

Again, as in the preceding case, we will let 
\[
\prod_{n=1}^{N} a_n^{n + \frac{n+2}{\log^{3+\varepsilon}\log a_n}} = \left(\prod_{n=1}^{k_2} a_n^{n + \frac{n+2}{\log^{3+\varepsilon}\log a_n}}\right) \left(\prod_{n=k_2+1}^{N} a_n^{n + \frac{n+2}{\log^{3+\varepsilon}\log a_n}}\right) = M_1 \cdot M_2.
\]
For $r < k_2$, $a_r \le a_{k_2} \le 2^{k_2}$, so as before
\[
M_1 \le \prod_{n=1}^{k_2} 2^{n^2 + n \frac{n+2}{\log^{3+\varepsilon}\log 2^n}} \le \prod_{n=1}^{k_2} 2^{n\left(k_2 + \frac{k_2+2}{\log^{3+\varepsilon}\log 2^{k_2}}\right)} \le 2^{k_2^3}.
\]

Now, by minimality of $N$, it follows for each $r\in(k_2,N]$ that
\[
a_r^{\frac{1}{D^n  n! \prod_{i=1}^{n-1}D_i}} \le \left(1+\frac{1}{n^{1+\varepsilon/4}}\right) \max_{j=k_2, \dots, n-1}  a_j^{\frac{1}{D^j  j! \prod_{i=1}^{j-1}D_i}}.
\]
Using this successively, we find that
\begin{alignat*}{2}
a_n^{\frac{1}{D^n  n! \prod_{i=1}^{n-1}D_i}}  &\le \left(1+\frac{1}{n^{1+\varepsilon/4}}\right) \left(1+\frac{1}{(n-1)^{1+\varepsilon/4}}\right) \max_{j=k_2, \dots, n-2}  a_j^{\frac{1}{D^j  j! \prod_{i=1}^{j-1}D_i}}\\ 
&\le \cdots \le a_{k_2}^{\frac{1}{D^{k_2}  k_2! \prod_{i=1}^{k_2-1}D_i}} \prod_{j=k_2+1}^{n} \left(1+\frac{1}{j^{1+\varepsilon/4}}\right)
\end{alignat*}
The latter is a partial product of a convergent infinite product, and so can be bounded by a constant depending only on $\varepsilon$. Since $a_{k_2} < 2^{k_2}$, the first factor is also bounded by $\sqrt{2}$, say. The upshot is that for some $B$ depending only on $\varepsilon$,
\[
a_n \le B^{D^n  n! \prod_{i=1}^{n-1}D_i},
\]
so that 
\[
M_2 \le \prod_{n=k_2+1}^{N} B^{\left(D^n  n! \prod_{i=1}^{n-1}D_i \right)\left(n + \frac{n+2}{\log^{3+\varepsilon}\log a_n}\right)} \le \prod_{n=k_2+1}^{N} B^{\left(D^n  n! \prod_{i=1}^{n-1}D_i \right)\left(n + \frac{n+2}{\log^{3+\varepsilon}(n\log 2)}\right)},
\]
since $a_n > 2^n$ by maximality of $k_2$.

Again as in the preceding case, consider now the sum
\[
\sum_{n=t}^{\infty} \frac{a_n^{\frac{1}{\log^{3+\varepsilon} \log a_n}}}{a_n} = \sum_{n=t}^{k_1-1} \frac{a_n^{\frac{1}{\log^{3+\varepsilon} \log a_n}}}{a_n} +\sum_{n=k_1}^{\infty} \frac{a_n^{\frac{1}{\log^{3+\varepsilon} \log a_n}}}{a_n} =S_1 + S_2.
\]

By Corollary \ref{cor:bound_series_an>2n} as before, 
\begin{equation}
\label{eq:Sum_1}
S_1 \le a_{N+1}^{\frac{1}{\log^{3+\varepsilon/2}\log a_{N+1}} - 1}
\end{equation}
Since 
\[
a_{N+1}^{\frac{1}{D^{N+1}  (N+1)! \prod_{i=1}^{N}D_i}} > \left(1+\frac{1}{(N+1)^{1+\varepsilon/4}}\right) \max_{j=s, \dots, N}  a_j^{\frac{1}{D^j  j! \prod_{i=1}^{j-1}D_i}},
\]
certainly,
\begin{equation}
\label{eq:Sum_2}
\begin{split}
a_{N+1} &>  \left(\left(1+\frac{1}{(N+1)^{1+\varepsilon/4}}\right) \max_{j=s, \dots, N}  a_j^{\frac{1}{D^j  j! \prod_{i=1}^{j-1}D_i}} \right)^{\frac{1}{D^{N+1}  (N+1)! \prod_{i=1}^{N}D_i}} \\
&\ge \left(1+\frac{1}{N^{1+\varepsilon/4}}\right)^{\frac{1}{D^{N+1}  (N+1)! \prod_{i=1}^{N}D_i}} \prod_{r=1}^{N} a_r.
\end{split}
\end{equation}
Furthermore, by minimality of $t$, for each $n\in(k_2,N]$, 
\begin{alignat*}{2}
a_n^{\frac{1}{D^n  n! \prod_{i=1}^{n-1}D_i}} &\le \left(1+\frac{1}{n^{1+\varepsilon/4}}\right) \max_{j=k_2, \dots, n-1}  a_j^{\frac{1}{D^j  j! \prod_{i=1}^{j-1}D_i}} \\
 &\le  \left(1+\frac{1}{n^{1+\varepsilon/4}}\right)\left(1+\frac{1}{(n-1)^{1+\varepsilon/4}}\right)  \max_{j=k_2, \dots, n-2}  a_j^{\frac{1}{D^j  j! \prod_{i=1}^{j-1}D_i}} \\
&\le \dots \\
&\le \prod_{j=k_2+1}^{n} \left(1+\frac{1}{j^{1+\varepsilon/4}}\right) a_{k_2}^{\frac{1}{D^j  j! \prod_{i=1}^{j-1}D_i}}.
\end{alignat*}
Since $a_{k_2} \le 2^{k_2}$, this is bounded by a constant, $K$ say, on estimating the product by its infinite counterpart. Consequenly, since $a_n$ is an increasing sequence by assumption,
\begin{equation}
\label{eq:Sum_3}
 \prod_{r=1}^{N} a_r =  \prod_{r=1}^{k_2} a_r  \prod_{r=k_2+1}^{N} a_r \le 2^{k_2^2} \prod_{r=k_2+1}^{N} K^{D^r  r! \prod_{i=1}^{r-1}D_i}
\end{equation}

We insert \eqref{eq:Sum_2}  into \eqref{eq:Sum_1} to obtain
\[
S_1 \le \left( \left(1+\frac{1}{N^{1+\varepsilon/4}}\right)^{\frac{1}{D^{N+1}  (N+1)! \prod_{i=1}^{N}D_i}} \prod_{r=1}^{N} a_r  \right)^{\frac{1}{\log^{3+\varepsilon/2}\log a_{N+1}} - 1}.
\]
Using \eqref{eq:Sum_0},
\[
S_1 \le \left( \left(1+\frac{1}{N^{1+\varepsilon/4}}\right)^{\frac{1}{D^t  t! \prod_{i=1}^{N}D_i}} \prod_{r=1}^{N} a_r  \right)^{\frac{1}{t^{3+ \varepsilon/4}} - 1}.
\]
Using finally \eqref{eq:Sum_3},
\[
S_1 \le \frac{
	\left(1+\frac{1}{N^{1+\varepsilon/4}}\right)^{\frac{1}{D^{N+1}  (N+1)! \prod_{i=1}^{N}D_i}\left( \frac{1}{N^{3+ \varepsilon/4}} - 1\right) } \left( 2^{k_2^2} \prod_{r=k_2+1}^{N} K^{D^r  r! \prod_{i=1}^{r-1}D_i}\right)^{\frac{1}{N^{3+ \varepsilon/4}}}
}{
	\prod_{r=1}^{N} a_r
}
\]

The second summand $S_2$ is estmated by Lemma \ref{lem:bound_series_general}, so that
\[
S_2 = \sum_{n=k_1}^{\infty} \frac{a_n^{\frac{1}{\log^{3+\varepsilon} \log a_n}}}{a_n} < \vert a_{k_1}\vert^{-\frac{\varepsilon}{2(1+\varepsilon)}} < A^{-\varepsilon(D^{k_1}  k_1! \prod_{i=1}^{k_1-1}D_i)/2(1+\varepsilon)}.
\]

In other words, since
\[ 
Z_{N} = (2^{N^2}M_1 M_2)^{D D_N}(S_1+ S_2),
\]
by inserting the estimates above, $Z_{N}$ can be made arbitrarily small by increasing $A$, which in turn corresponds to increasing $t$, so that in this case the \emph{liminf} of the Lemma is also equal to zero. The four cases exhaust the possibilities of satisfying the conditions of the lemma, and so the proof is complete.
\end{proof}

\section{Concluding Remarks}
In the light of Theorem \ref{thm:main1}, we expect that Theorem \ref{thm:main2} will remain true if the \emph{limsup} criterion \eqref{eq:limsup=infty} is weakened to the assumption that $\lvert\alpha_{N,1}\rvert^{\frac{1}{ D^{N}N! \prod_{n=1}^{N-1}D_n}}$ diverges in $\mathbb{R}$, though this will require additional arguments in lemma \ref{lem:upper bound 2}, most likely in the form of two additional cases to be considered.
We deemed this question out of scope for the current paper, however.

Similarly, the proof of Theorem \ref{thm:main2} may be modified so that the bound on $b_{n}\house{a_n}$ can be loosened to $b_n\le |a_n|^{(\log\log |a_n|)^{-3-\varepsilon}}$ and $\house{a_n}\le |a_n|^{1+ (\log\log |a_n|)^{-3-\varepsilon}}$, thus presenting the same lenient bound on $b_n$ as found in \cite{MR2851961}. Note however that in order to accomplish this, we would then have to strengthen the divergence assumption to
\begin{equation*}
	\limsup_{n\to\infty}|a_n|^{1/D^n\prod_{i=1}^{n-1}(D_i+d_i)} = \infty,
\end{equation*}
at least until the case of
\begin{equation*}
	\liminf_{n\to\infty}|a_n|^{1/D^n\prod_{i=1}^{n-1}(D_i+d_i)} <\limsup_{n\to\infty}|a_n|^{1/D^n\prod_{i=1}^{n-1}(D_i+d_i)} < \infty
\end{equation*}
has been handled.

Furthermore, in \cite{Laursen2024}, an analogue of Theorem \ref{thm:main1} for series $\sum_{n=1}^{\infty}\frac{b_n}{\alpha_n}$ with $D_n=d$ constant was proven (Proposition 4.3 of the paper) and where the divergence criterion is replaced by \emph{limsup} criterion.
Compared to the present paper and \cite{MR4022087,Laursen}, the exponent in the \emph{limsup} expression is not
\begin{align*}
	\limsup_{n\to\infty} |\alpha_n|^{D^{-n}(1+d)^{-n}} = \infty,
\end{align*}
as would be expected, but rather
\begin{align*}
	\limsup_{n\to\infty} |\alpha_n|^{(1+dD)^{-n}} = \infty.
\end{align*}
We therefore suspect that Theorem \ref{thm:main1} may be improved so that the exponent in the divergence criterion may be replaced with
\begin{align*}
	\frac{1}{\prod_{i=1}^{n-1}(d_i+DD_i)}.
\end{align*}
This is less strict when $D>1$. We further suspect tha the exponent in the \emph{limsup} expression of Theorem \ref{thm:main2} may be replaced with
\begin{equation*}
	\frac{1}{\prod_{i=1}^{N-1}(d_i+iDD_i)},
\end{equation*}
which is easily checked to be more lenient when $DD_i>1$.

It is also likely that the restrictions on real and imaginary values in Theorem \ref{thm:main2} -- including the alternative restrictions presented in the subsequent remark -- may be weakened to some extent.

\providecommand{\bysame}{\leavevmode\hbox to3em{\hrulefill}\thinspace}
\providecommand{\MR}{\relax\ifhmode\unskip\space\fi MR }
\providecommand{\MRhref}[2]{%
  \href{http://www.ams.org/mathscinet-getitem?mr=#1}{#2}
}
\providecommand{\href}[2]{#2}

\end{document}